\documentclass[11pt,leqno]{amsart}
\usepackage{amsmath}
\usepackage{eucal}
\usepackage{amsfonts}
\usepackage{amssymb}
\usepackage{geometry}
\usepackage{graphicx}
\usepackage{hyperref}
\newtheorem{theorem}{Theorem}

\newtheorem{corollary}[theorem]{Corollary}
\newtheorem{definition}[theorem]{Definition}

\newtheorem{lemma}[theorem]{Lemma}

\newtheorem{proposition}[theorem]{Proposition}
\newtheorem{remark}[theorem]{Remark}

\newtheorem*{theorem*}{Theorem}
\newtheorem*{proposition*}{Proposition}
\newtheorem*{corollary*}{Corollary}

\newtheorem{theoremA}{Theorem}
\newtheorem{corollaryA}[theoremA]{Corollary}

\newcommand{\R}{\mathbb{R}}
\newcommand{\Ric}{\mathrm{Ric}}
\newcommand{\Hess}{\mathrm{Hess}}
\renewcommand{\div}{\mathrm{div}}
\newcommand{\CN}{\mathcal{N}}
\newcommand{\CO}{\mathcal{O}}
\newcommand{\vol}{\mathrm{vol}}
\newcommand{\g}{\langle \cdot, \cdot \rangle}

\def\dd{\mathrm{d}}

\def\<{\langle}
\def\>{\rangle}
\def\bea{\begin{eqnarray*}}
\def\eea{\end{eqnarray*}}
\def\be{\begin{equation}}
\def\ee{\end{equation}}

\def\qed{\ifhmode\unskip\nobreak\fi\ifmmode\ifinner\else
\hskip5 pt \fi\fi\hbox{\hskip5 pt \vrule width4 pt height6 pt
depth1.5 pt \hskip 1pt }}

\begin{document}

\title{Height estimates for Killing graphs}
\author{Debora Impera}
\address{Dipartimento di Matematica e Applicazioni\\
Universit\`a di Milano--Bicocca Via Cozzi 53\\
I-20125 Milano, ITALY}
\email{debora.impera@gmail.com}
\author{Jorge H. de Lira}
\address{Departamento de Matem\'atica\\Universidade Federal do Cear\'a-UFC\\60455-760 Fortaleza, CE, Brazil.}
\email{jorge.lira@mat.ufc.br}
\author{Stefano Pigola}
\address{Sezione di Matematica - DiSAT\\
Universit\'a dell'Insubria - Como\\
via Valleggio 11\\
I-22100 Como, Italy}
\email{stefano.pigola@uninsubria.it}
\author{Alberto G. Setti}
\address{Sezione di Matematica - DiSAT\\
Universit\'a dell'Insubria - Como\\
via Valleggio 11\\
I-22100 Como, Italy}
\email{alberto.setti@uninsubria.it}
\date{\today}
\maketitle

\begin{abstract}
The paper aims at proving global height estimates for Killing graphs defined over a complete manifold with nonempty boundary. To this end, we first point out how the geometric analysis on a Killing graph is naturally related to a weighted manifold structure, where the weight is defined in terms of the length of the Killing vector field. According to this viewpoint, we introduce some potential theory on weighted manifolds with boundary and we prove a weighted volume estimate for intrinsic balls on the Killing graph. Finally, using these tools, we provide the desired estimate for the weighted height function in the assumption that the Killing graph has constant weighted mean curvature and the weighted geometry of the ambient space is suitably controlled.
\end{abstract}


\section*{Introduction and main results}

Let $\left(M, \g_M\right)  $ be a complete, $(n+1) $-dimensional
Riemannian manifold endowed with a complete Killing vector field $Y$
whose orthogonal distribution has constant rank $n$ and it is integrable. Let
$(P,  \g_P)$ be an integral leaf of that distribution
equipped with its induced complete Riemannian metric $\g_P$. The flow $\vartheta: P \times \R \rightarrow M$ generated by $Y$ is an isometry between $M$ and the warped product $P\times_{e^{-\psi}}\R $ with metric
\[
\g_{M}= \g_P+e^{-2\psi}\dd s \otimes \dd s
\]
where $s$ is the flow parameter and $\psi = - \log |Y|$.

Let $\Omega\subset P$ be a possibly unbounded domain with regular boundary
$\partial\Omega\neq\emptyset$.
The {\it Killing graph} of a smooth function $u:\bar{\Omega}\rightarrow
\R $ is the hypersurface $\Sigma\subset M$ parametrized by the map
\[
X(x)=\vartheta(x,u(x)), \quad x\in \bar\Omega.
\]
Obviously, if $Y$ is a parallel vector field, then
$M$ is isometric with the Riemannian product $P\times\R $ and the notion
of a Killing graph reduces to that of a usual vertical graph.\smallskip

The above terminology, together with some existence results, was first introduced by M. Dajczer,  P.A. Hinojosa, and J.H. de Lira in \cite{DHL-CalcVar}. Since then, Killing graphs have become the subject of a systematic investigation both in order to understand their geometry and as a tool to study different problems such as the existence of solutions of the asymptotic Plateau problem in certain symmetric spaces; \cite{CR-Asian, Ri-preprint}.\medskip

The aim of this paper is to obtain quantitative height estimates for a smooth
Killing graph
\[
\Sigma=\mathrm{Graph}_{\bar \Omega}\left(  u\right)  \hookrightarrow M=P\times_{e^{-\psi}}\R
\]
parametrized over (the closure of) a possibly unbounded domain $\Omega\subset P$ and whose
smooth boundary $\partial\Sigma\neq\emptyset$ is contained in the totally
geodesic slice $P\times\{0\}$ of $M$.\medskip

When $\psi (x)  \equiv \mathrm{const}$ and the ambient manifold is
the Riemannian product $P\times\R $, it is well understood that
quantitative a-priori estimates can be deduced by assuming that the mean
curvature $H$ of the graph is constant (\textit{CMC graphs} for short). For
bounded domains into the Euclidean plane $P=\R ^{2}$ this was first
observed in seminal papers by E. Heinze, \cite{He}, and J. Serrin, \cite{Se}. More precisely, assume
that $H>0$ with respect to the\textit{ downward pointing Gauss map} $\CN$. Then, $\Sigma$ is confined into the slab $\R ^{2}\times
\lbrack0,1/H]$ regardless of the size of the domain $\Omega$. This type of
estimates has been recently extended to unbounded domains $\Omega
\subset\R ^{2}$ by A. Ros and H. Rosenberg, \cite{RoRo-AJM}. Their
technique, which is based on smooth convergence of CMC surfaces, requires
strongly that the base leaf $P=\R ^{2}$ is homogeneous and cannot be
trivially adapted to general manifolds. In the case of a generic base manifold
$P$, and maintaining the assumption that the CMC vertical graph $\Sigma$ is
parametrized over a bounded domain, the corresponding height estimates have
been obtained by D. Hoffman, J.H. de Lira and H. Rosenberg, \cite{HLR-TAMS}, J.A.
Aledo, J.M. Espinar and J.A. Galvez, \cite{AEG-Illinois}, L. Al\'\i as and M.
Dajczer, \cite{AD-PEMS} etc. The geometry of $P$ enters the game in the form
of curvature conditions, namely, the Ricci curvature of \ $P$ cannot be too
much negative when compared with $H$. In particular, for non-negatively Ricci
curved bases the height estimates hold with respect to any choice of $H>0$. In
the very recent \cite{IPS-Crelle}, the boundedness assumption on the domain
$\Omega$ has been replaced by a quadratic volume growth condition, thus
obtaining a complete extension of the Ros-Rosenberg result to any complete
manifold $P$ with non-negative sectional curvature and dimensions $n\leq4$.
The restriction on the dimension is due to the fact that, up to now, it is not
known whether CMC graphs over non-negatively curved manifolds $P$ are
necessarily contained in a vertical slab. Granted this, the desired estimates
can be obtained.\medskip

In a slightly different perspective, qualitative bounds of the height of CMC
vertical graphs on bounded domains have been obtained by J. Spruck,
\cite{S-PAMQ}. It is worth to observe that his technique, based on Serrin-type
gradient estimates and Harnack inequalities, is robust enough to give a-priori
bounds even in the case where the mean curvature is non-constant. Actually, it works
even for Killing graphs up to using the work by M. Dajczer and J.H. de Lira,
\cite{DL-Poincare}; see also \cite{DLR-JAM}. However, in the Killing setting, the problem of obtaining
quantitative bounds both on bounded and\ on unbounded domains remained
open.\medskip

Due to the structure of the ambient space, it is reasonable to expect that
an a-priori estimate for the height function of a CMC Killing graph is
sensitive of the deformation function ${\psi}$. In fact, since the length
element of the fibre $\left\{  x\right\}  \times\R $ is weighted by the
factor $e^{-\psi(x)}$, a reasonable pointwise bound should
be of the form
\[
0 \leq e^{-\psi(x)}u(x)  \lesssim \frac{1}{|H|}.
\]
Actually, the same weight $e^{-\psi(x)}$ appears also in the expression of the volume element of $\Sigma$ thus suggesting the existence of an intriguing interplay between Killing graph and smooth metric measure spaces (also called {\it weighted manifolds}). Since this interplay represents the leading idea of the entire paper we are going to take a closer look at how it arises.\medskip

In view of the fact that we are considering graphical hypersurfaces, weighted structures should appear both at the level of the base manifold $P$, where $\Sigma$ is parametrised, and at the level of the ambient space $M$, where $\Sigma$ is realised. In fact, these two weighted contexts will interact in the formulation of the main result.
To begin with, we note that the induced metric on $\Sigma = \mathrm{Graph}_{\bar \Omega}(u)$ is given by
\begin{equation}
\g_{\Sigma} = \g_P+e^{-2\psi}\dd u\otimes \dd u.\label{1st}
\end{equation}
Thus, the corresponding Riemannian volume element $\dd \Sigma$  has the expression
\begin{equation}\label{volumelement}
\mathrm{d}\Sigma= W e^{-\psi}\dd P,
\end{equation}
where $W= \sqrt{e^{2\psi} +|\nabla^Pu|^2}$ and $\dd P$ is the volume element of $P$.
As alluded to above, the special form of \eqref{volumelement}, when compared with the case of a product ambient space, suggests to switch the viewpoint from that of the Riemannian manifold $(P,\g_{P})$ to that of the smooth metric measure space
\[
P_{\psi}:= (P,\g_{P},\dd P_{\psi})
\]
where we are using the standard notation
\[
\dd P_{\psi} = e^{-\psi} \dd P.
\]
In particular,
\[
\dd \Sigma = W \dd P_{\psi}
\]
and we are naturally led to investigate to what extent the geometry of $\Sigma$ is influenced by the geometry of the weighted space $P_{\psi}$. As we shall see momentarily, the geometry of $P_{\psi}$ will enter the game in the form of a growth condition on the weighted volume of its geodesic balls $B^{P}_{R}(o)$:
\[
\vol_{\psi} (B^{P}_{R}(o))  = \int_{B^{P}_{R}(o)} \dd P_{\psi}.
\]
In a different direction, we observe that the smooth metric measure space structure of $P_{\psi}$  extends to the whole ambient space up to  identifying $\psi : P \to \R$ with the function $\bar \psi : P \times_{e^{-\psi}} \R \to \R$ given by
\[
\bar \psi (x,s) = \psi (x).
\]
With a slight abuse of notation, we write
\begin{equation}
 M_{\psi} := (M, \g_{M}, \dd_{\psi}M) = (P \times_{e^{-\psi}} \R , e^{-\psi} \dd M)
\end{equation}
and we can consider the original Killing graph as an  hypersurface
\[
\Sigma = \mathrm{Graph}_{\bar \Omega}(u) \hookrightarrow M_{\psi}.
\]
Previous works on classical height estimates for CMC graphs show that the relevant geometry of the ambient space is subsumed to a condition on its Ricci tensor. Thus, if we think of realizing $\Sigma$ inside $M_{\psi}$ we can expect that height estimates need a condition on its Bakry-\'Emery Ricci tensor defined by
\[
\Ric^{M}_{\psi} = \Ric^{M} + \Hess^{M}(\psi).
\]
We shall come back to this later on.
Following \cite{DHL-CalcVar}, we now orient
$\Sigma$ using the {\it upward pointing} unit normal
\begin{equation} \label{N}
\CN = \frac{e^{2\psi} Y - \vartheta_* \nabla^P u}{\sqrt{e^{2\psi}+|\nabla^P u|^2}}=\frac{1}{W}\big(e^{2\psi} Y - \vartheta_* \nabla^P u\big).
\end{equation}
Note that $\CN$ is upward pointing in the sense that
\begin{equation}\label{NY}
\langle \CN, Y\rangle_{M}=\frac{1}{W} >0.
\end{equation}
Let $H:\Omega\subseteq P \to \R $ be the corresponding mean curvature function.
The weighted $n$-volume
associated to (the restriction of) $\psi$ (to $\Sigma$) is defined by
\begin{equation}
\label{Apsi}
\mathcal{A}_{\psi}[\Sigma] := \int_\Sigma \dd \Sigma_{\psi}.
\end{equation}
We are not concerned with the convergence of the integral.
Given a compactly supported variational vector field $Z$ along $\Sigma$ the first variation formula reads
\begin{equation}
\delta_Z \mathcal{A}_{\psi} = \int_\Sigma \left(\div ^\Sigma Z -
\langle  {\nabla}^{M} \psi, Z\rangle_{M}\right) \dd \Sigma_{\psi}.
\end{equation}
In particular, if $Z = v \CN$ for some $v\in C_{c}^\infty(\Sigma)$ we have
\begin{equation}
\delta_Z \mathcal{A}_{\psi} = \int_\Sigma \left(-nH - \langle  \nabla^{M} \psi, \CN\rangle_{M}\right) v \, \dd \Sigma_{\psi}= -n\int_\Sigma H_{\psi} v \, \dd \Sigma_\psi,
\end{equation}
where, using the definition  proposed  by M. Gromov, \cite{Gro},
\begin{equation}
\label{Hpsi}
H_{\psi} = H + \frac1n \langle  {\nabla}^{M} \psi, \CN\rangle_{M}
\end{equation}
is the {\it $\psi$-weighted mean curvature} of $\Sigma$.\smallskip

The way we have followed to introduce the weighted structure on the ambient space $M$ may look the most natural: it is trivially compatible with the weighted structure of the base space $P_{\psi}$ and with the weighted height function of $\Sigma$. Moreover, the weight $\psi$  appears in the volume element of $\Sigma$. However, it is worth to note that this is not the only ``natural'' choice. This becomes clear as soon as we express the mean curvature (and its modified version) of $\Sigma = \mathrm{Graph}_{\Omega}(u)$ in the classical form of a capillarity equation. Indeed, it is  shown in \cite{ DL-Poincare, DHL-CalcVar}  that
\begin{equation}\label{capillary-2}
\div ^P\Big(\frac{\nabla^P u}{W}\Big) =n H_{-\psi}.
\end{equation}
which, in view of \eqref{Hpsi}, is completely equivalent to
\begin{equation}\label{capillary-3}
\div _{\psi}^P\left(\frac{\nabla^P u}{W} \right)= nH.
\end{equation}
Here, we are using the standard notation for the {\it weighted divergence} in $P_{\psi}$:
\[
\div^{P}_{\psi}X \, = e^{\psi} \div (e^{-\psi}  X) = \div^{P} X - \langle X, \nabla \psi\rangle_{P}.
\]
Thus, with respect to the capillarity type equation \eqref{capillary-2}, the natural and relevant weighted structure on $M$ arises from the weight $-\psi$ and we might be led to consider $\Sigma$ as an hypersurface in $M_{-\psi}$. On the other hand, the original choice $\psi$ fits very well into the  capillarity type equation \eqref{capillary-3}. Both these weighted structures are relevant and a choice has to be made. To give an idea of this kind of duality between $\psi$ and $-\psi$ structures, we extend, in the setting of Killing graphs, the classical relation between the mean curvature of the graph and the isoperimetric properties of the parametrization domain. This is the content of the following  weighted versions of a result by E. Heinz, \cite{He2}, S.S. Chern, \cite{Ch}, H. Flanders, \cite{Fl}, and I. Salavessa \cite{Salavessa-PAMS}.
Define the ``standard'' and  the``weighted'' Cheeger constants of a domain $\Omega$ by, respectively,
\begin{equation}\nonumber
\mathfrak{b}(\Omega) = \inf_{D} \frac{\vol(\partial D)}{\vol(D)},
\quad
\mathfrak{b}_\psi (\Omega) = \inf_{D} \frac{\vol_\psi(\partial D)}{\vol_\psi(D)},
\end{equation}
where $D$ is a bounded subdomain with compact closure in $\Omega$ and with regular boundary $\partial D \not= \emptyset$.
\begin{proposition*}
Let $\Sigma = \mathrm{Graph}_{\Omega}(u) \hookrightarrow P \times_{e^{-\psi}} \R$ be an $n$-dimensional Killing graph defined over the domain $\Omega\subset P$. Then, the mean curvature $H$ of $\Sigma$, and its weighted version $H_{-\psi}$, satisfy the following inequalities:
\begin{equation}
n\inf_\Omega |H_{-\psi}|\le \mathfrak{b}(\Omega), \quad n\inf_\Omega |H|\le \mathfrak{b}_\psi (\Omega).
\end{equation}
In particular:
\begin{itemize}
 \item [(i)] If $\Omega \subset P_{\psi}$ has zero weighted Cheeger constant, and $\Sigma$ has constant mean curvature $H$,
then $\Sigma$ is a minimal graph.
\item [(ii)] If $\Omega \subset P$ has zero Cheeger constant, and $\Sigma \hookrightarrow M_{-\psi}$ has constant weighted mean curvature $H_{-\psi}$, then $\Sigma$ is a $(-\psi)$-minimal graph.
\end{itemize}

\end{proposition*}
Indeed, if $D \Subset P$ is a relatively compact domain in $P$ with boundary $\partial D = \Gamma$ and outward pointing unit normal $\nu_0\in TP$, integrating \eqref{capillary-3} and using the weighted version of the divergence theorem
\[
\int_{D} \div_{\psi} Z\,  \dd P_{\psi} = \int_{\Gamma} \langle Z , \nu_{0} \rangle_{P}\,  d\Gamma_{\psi}
\]
we obtain
\[
\int_{D} nH \dd P_{\psi} = \int_{\Gamma} \left\langle \frac{\nabla^P u}{W},\nu_0\right\rangle_{P} \, \dd \Gamma_{\psi}
\]
If $H$ has constant sign, multiplying by $-1$ if necessary, we deduce that
\begin{equation*}
\int_D n|H| \dd P_{\psi}  \leq \int_\Gamma \left\vert \left\langle \frac{\nabla^P u}{W},\nu_0 \right\rangle \right\vert  \dd \Gamma_{\psi}
\end{equation*}
and recalling that $|\nabla^P u|/W\leq 1$, we conclude that
\[
n\inf_D |H| \int_D \dd P_{\psi} \leq \int_\Gamma \dd \Gamma_{\psi}.
\]
that is
\begin{equation}\label{Salavessa_weighted}
 n \inf_D |H| \vol_{\psi}(D) \leq \vol_{\psi}(\partial D).
\end{equation}
Note that  \eqref{Salavessa_weighted} certainly holds even if $H$ has not constant sign, for then $\inf |H| =0$. In a completely similar way, starting from equation \eqref{capillary-2}, we obtain
\begin{equation}\nonumber
 n \inf_{D} |H_{\psi}| \vol(D) \leq \vol(\partial D).
\end{equation}
The desired conclusions now follow trivially.\medskip

This brief discussion should help to put in the appropriate perspective the following theorem which represents the main result of the paper.

\begin{theoremA} \label{th_fheightestimate}
Let $(M,\g_{M})  $ be a complete, $(n+1)$-dimensional Riemannian manifold endowed with a complete Killing vector field $Y$ whose orthogonal distribution has constant rank $n$ and it is integrable. Let $(P,\g_{P})$ be an integral leaf of that distribution and let $\Omega\subset P$ be a smooth domain. Set $\psi=-\log(|Y|)$ and assume that:

\begin{enumerate}
\item[(a)] $-\infty<\inf_{{\Omega}} \psi \leq \sup_{{\Omega}} \psi<+\infty$;

\item[(b)] $\mathrm{vol}_{\psi}\left(  B_{R}^{P}\cap\Omega\right)  =\mathcal{O}(R^2)  $,
as $R\rightarrow+\infty$;

\item[(c)] $\Ric^{M}_\psi\geq 0$ in a neighborhood of $\Omega\times\R  \subset M_{\psi}$.

\end{enumerate}
Let $\Sigma=\mathrm{Graph}_{\bar \Omega}(u)$ be a Killing graph over $\bar \Omega$ with weighted mean curvature $H_{\psi} \equiv \mathrm{const} < 0$ with respect to the upward pointing Gauss map $\CN$. Assume that:

\begin{enumerate}
\item[(d)] The boundary $\partial\Sigma$ of $\Sigma$ lies in the slice
$P\times\{0\}$;

\item[(e)] The weighted height function of $\Sigma$ is bounded: $\sup_{\Sigma} |u e^{-\psi}|< +\infty$.
\end{enumerate}

\noindent Then
\begin{equation}
\label{hest}
0\leq u(x) e^{-\psi(x)}  \leq\frac{C}{|H_{\psi}|},
\end{equation}
where $C:=e^{2(\sup_\Omega \psi-\inf_\Omega\psi)}\geq 1$.
\end{theoremA}
It is worth pointing out that the constant $C$ in \eqref{hest} depends only on the variation  of $\psi$ .

The strategy of proof follows the main steps in \cite{IPS-Crelle}. For instance, in order to to get the upper estimate, we will use potential theoretic properties of the weighted manifold with boundary $\Sigma_{\psi} = (\Sigma, \g_{\Sigma}, \dd \Sigma_{\psi})$.
The idea is to show that, thanks to the capillarity equation \eqref{capillary-2}, the moderate volume growth assumption (b) on $\Omega$ is inherited by $\Sigma_{\psi}$. Therefore, the weighted Laplacian, defined on a smooth function $w : \Sigma \to \R$ by
\[
\Delta^{\Sigma}_{\psi} w = \div^{\Sigma}_{\psi}(\nabla^{\Sigma} w) = \Delta^{\Sigma} w - \langle \nabla^{\Sigma} \psi, \nabla^{\Sigma} w \rangle_{\Sigma},
\]
satisfies a global maximum principle similar to that valid on a compact set. And it is precisely in the compact setting that, throughout the construction of explicit examples, we shall show that our height estimate is essentially sharp. Moreover, in the spirit of  known results in the product case (see above), since we do not impose any restriction on the size of the mean curvature, the ambient space $M_{\psi}$ is assumed to have non-negative weighted (i.e. Bakry-\'Emery) Ricci curvature.
Note however, see Remark~\ref{RmkHLR}, that the result extends to the case where $\Ric^{M}_\psi$ is bounded below by a negative constant, provided a suitable bound on $H^2$ is imposed.

On the other hand, to show that $u\geq 0$ one uses the parabolicity of  $\Delta^\Sigma_{3\psi}$. This provides further instance of the  interplay between different weight structures on $M.$

It is worth to point out that, as it often happens in submanifolds theory, the case $n=2$ is very special. In fact, for two dimensional Killing graphs, one can show that the curvature assumption (c) implies the boundedness condition (e). Therefore, the previous result takes the following striking form.

\begin{corollaryA} \label{coro_fheightestimate}
Let $(M,\g_{M})  $ be a complete, $3$-dimensional
Riemannian manifold endowed with a complete Killing vector field $Y$ whose orthogonal
distribution has constant rank $2$ and it is integrable. Let $(P,\g_{P})$ be an integral leaf of that distribution and let $\Omega\subset P$ be a smooth domain. Set $\psi=-\log(|Y|)$ and assume that:
\begin{enumerate}
\item[(a)] $-\infty<\inf_{{\Omega}} \psi \leq \sup_{{\Omega}} \psi<+\infty$;

\item[(b)] $\vol_{\psi}\left(  B_{R}^{P}\cap\Omega\right)  = \CO(R^2)  $,
as $R\rightarrow+\infty$;

\item[(c)] $\Ric^{M}_\psi\geq 0$ in a neighborhood of $\Omega\times\R $.

\end{enumerate}
Let $\Sigma=\mathrm{Graph}_{\Omega}(u)$ a be a $2$-dimensional Killing graph over $\Omega$ with constant  weighted mean curvature $H_{\psi} \equiv \mathrm{const}<0$ with respect to the upward pointing Gauss map and with boundary $\partial \Sigma \subset P\times\{0\}$.

\noindent Then
\[
0\leq u e^{-\psi} (x)  \leq\frac{C}{|H_{\psi}|},
\]
where $C:=e^{2(\sup_\Omega \psi-\inf_\Omega\psi)}\geq 1$.
\end{corollaryA}
\bigskip

The organization of the paper is as follows:\smallskip

\noindent In Section \ref{PotTheory} we prove some potential theoretic properties of weighted manifolds with boundary, related to global maximum principles for the weighted Laplacian. These results, using new direct arguments, extend to the weighted context previous investigation in \cite{IPS-Crelle}.\smallskip

\noindent Section \ref{HeightEst} contains the proof of the quantitative height estimates for Killing graphs over possibly unbounded domains. To this end, we shall introduce: (i) some basic formulas for the weighted Laplacian of the height and the angle functions of the Killing graph and (ii) weighted volume growth estimate that will enable us to apply the global maximum principles obtained in Section \ref{PotTheory}.\smallskip

\noindent In Section \ref{section-examples} we construct concrete examples of Killing graphs with constant weighted mean curvature which, in particular, show that the constant $C$ in our estimate has the correct functional dependence on the variation $\sup \psi-\inf \psi$ of $\psi$.\\

\noindent{\bf Acknowledgments.}
 It is  our pleasure to thank the anonymous referee for a very careful reading of the manuscript and for important  suggestions that greatly improved the exposition.

\section{Some potential theory on weighted manifolds}\label{PotTheory}
The aim of this section is to study potential theoretic properties, and, more precisely, parabolicity with respect to the weighted Laplacian, of a weighted Riemannian manifold with possibly empty boundary. This relies on  the notion of weak sub (super) solution subject to Neumann boundary conditions that we are going to introduce. In order to avoid confusion, and since we are dealing with general results valid on any weighted manifold, throughout this Section we shall call $f:M \to \R$ the weight function. The symbol $\psi$ is thus reserved to the peculiar weight related to Killing graphs.\medskip

Let $M_f = (M,\g,e^{-f}\dd M)$ be a smooth, $n$-dimensional, weighted manifold with smooth
boundary $\partial M\neq\emptyset$ oriented by the exterior unit normal $\nu.$ The interior of $M$ is denoted by $\mathrm{int}M = M \setminus \partial M$. By a domain in $M$ we mean a non-necessarily connected open set $D\subseteq M$. We say
that the domain $D$ is smooth if its topological boundary $\partial D$ is a
smooth hypersurface $\Gamma$ with boundary $\partial\Gamma=\partial
D\cap\partial M$. Adopting the notation in \cite{IPS-Crelle}, for any domain
$D\subseteq M$ we define
\begin{align*}
\partial_{0}D&=\partial D\cap\mathrm{int}M,\\
\partial_{1}D&=\partial M\cap D
\end{align*}
We will refer to $\partial_0 D$ and to $\partial_1 D$ respectively as the \textit{Dirichlet boundary} and the \textit{Neumann boundary} of the domain $D$. Finally, the \textit{interior part} of $D$, in the sense of manifolds with boundary, is defined as
\[
 \mathrm{int}D=D \cap  \mathrm{int}M,
\]
so that, in particular,
\[
D= \mathrm{int}D \cup \partial_1 D.
\]
\medskip

We recall that the Sobolev space $W^{1,2}( \mathrm{int}D_f)$ is defined as the Banach space of functions $u \in L^2( \mathrm{int}D_f)$ whose distributional gradient satisfies $\nabla u \in L^2( \mathrm{int}D_f)$. Here we are using the notation
\[
L^2( \mathrm{int}D_f):= \{w\, :\, \int_{D} w^2\dd M_f<+\infty\}.
\]
By the Meyers-Serrin density result, this space coincides with the closure of $C^{\infty}(\mathrm{int}D)$ with respect to the Sobolev norm $\|u\|_{W^{1,2}(M_f)}=\|u\|_{L^2(M_f)}+\|\nabla u\|_{L^2(M_f)}$.
Moreover, when $D=M$ and $M$ is complete, $W^{1,2}(\mathrm{int}M_f)$ can be also realised as the $W^{1,2}(M_f)$-closure of $C^{\infty}_{c}(M)$.

Finally, the space $W^{1,2}_{\mathrm{loc}}(\mathrm{int}D_f)$ is defined by the condition that $u \cdot \chi \in W^{1,2}(\mathrm{int}D_f)$ for every cut-off function $\chi \in C^{\infty}_c(\mathrm{int}D)$. We extend this notion by including the Neumann boundary of the domain as follows
\[
W^{1,2}_{\mathrm{loc}}(D_f)=\{u \in W^{1,2}(\mathrm{int}\Omega_f), \text{ }\forall \text{domain } \Omega \Subset  D=\mathrm{int}D \cup \partial_1 D\}.
\]

Now, suppose $D\subseteq M$ is any domain. We put the following Definition. Recall from the Introduction that the $f$-Laplacian of $M_{f}$ is the second order differential operator defined by
\[
\Delta_{f} w = e^{f} \div(e^{-f} \nabla w) = \Delta w - \langle \nabla f , \nabla w \rangle,
\]
where $\Delta$ denotes the Laplace-Beltrami operator of $(M,\g)$. Clearly, as one can verify from the weighted divergence theorem, $-\Delta_{f}$ is a non-negative, symmetric operator on $L^{2}(M_{f})$.
\begin{definition}
By a weak Neumann sub-solution $u\in W_{loc}^{1,2}\left(D_f\right)  $ of the $f$-Laplace equation, i.e., a weak solution of the problem%
\begin{equation}
\left\{
\begin{array}
[c]{ll}%
\Delta_f u\geq0 & \text{on }\mathrm{int}D\\
\dfrac{\partial u}{\partial\nu}\leq0 & \text{on }\partial_{1}D,
\end{array}
\right.  \label{subneumannproblem}%
\end{equation}
we mean that the following inequality%
\begin{equation}
-\int_{D}\left\langle \nabla u,\nabla\varphi\right\rangle \dd M_f\geq0 \label{fsubsol}%
\end{equation}
holds for every $0\leq\varphi\in C_{c}^{\infty}\left(  D\right)  $.
Similarly, by taking $D=M$, one defines the notion of weak Neumann subsolution of the $f$-
Laplace equation on $M_f$ as a function $u\in W_{loc}^{1,2}\left(  M_f\right)  $
which satisfies (\ref{fsubsol}) for every $0\leq\varphi\in C_{c}^{\infty
}\left(  M\right)  $. As usual, the notion of weak supersolution can be obtained by reversing the inequality and, finally, we speak of a weak solution when the equality holds in (\ref{fsubsol}) without any sign condition on $\varphi$.
\end{definition}

\begin{remark}
\rm{
Analogously to the classical case, in the above definition, it is equivalent to require that
(\ref{fsubsol}) holds for every $0\leq\varphi\in \mathrm{Lip}_{c}\left(  D\right)  $.
}
\end{remark}

Following the terminology introduced in \cite{IPS-Crelle}, we are now ready to give the following
\begin{definition}
\label{def_fparab} A weighted manifold $M_f$ with boundary $\partial
M\neq\emptyset$ oriented by the exterior unit normal $\nu$ is said to be $f$-parabolic if any bounded above, weak Neumann
subsolution of the $f$-Laplace equation on $M_f$ must be constant. Namely, for
every $u\in C^{0}\left(  M\right)  \cap W_{loc}^{1,2}\left(  M_f\right)  $,%
\begin{equation}
\label{def_fpar}%
\begin{array}
[c]{ccc}%
\left\{
\begin{array}
[c]{ll}%
\Delta_f u\geq0 & \text{on }\mathrm{int}M\\
\dfrac{\partial u}{\partial\nu}\leq0 & \text{on }\partial M\\
\sup_{M}u<+\infty &
\end{array}
\right.  & \Rightarrow & u\equiv\mathrm{const}.
\end{array}
\end{equation}

\end{definition}

In the boundary-less setting it is by now well-known that $f$--parabolicity is related to a wide class of equivalent properties
involving the recurrence of the Brownian motion, $f$--capacities of condensers,
the heat kernel associated to the drifted laplacian, weighted volume growth,
function theoretic tests, global divergence theorems and many other geometric and potential-analytic properties. All these characterization can be proven to hold true also in case of weighted manifolds with non-empty boundaries. However, here we limit ourselves to pointing out the following characterization.

\begin{theorem}\label{thm_fAhlfors}
A weighted manifold $M_f$ is $f$-parabolic if and
only the following maximum principle holds. For every domain $D\subseteq M$
with $\partial_{0}D\neq\emptyset$ and for every $u\in C^{0}\left(
\overline{D}\right)  \cap W_{loc}^{1,2}\left(  D_f\right)  $ satisfying
\[
\left\{
\begin{array}
[c]{ll}%
\Delta_f u\geq0 & \text{on }\mathrm{int}D\\
\dfrac{\partial u}{\partial\nu}\leq0 & \text{on }\partial_{1}D\\
\sup\limits_{D}u<+\infty &
\end{array}
\right.
\]
in the weak sense, it holds%
\[
\sup_{D}u=\sup_{\partial_{0}D}u.
\]
Moreover, if $M_f$ is a $f$-parabolic manifold with boundary
$\partial M\neq\emptyset$ and if $u\in C^{0}\left(  M\right)  \cap W_{loc}%
^{1,2}\left(  \mathrm{int}M_f\right)  $ satisfies%
\[
\left\{
\begin{array}
[c]{ll}%
\Delta_f u\geq0 & \text{on }\mathrm{int}M\\
\sup_{M}u<+\infty &
\end{array}
\right.
\]
then%
\[
\sup_{M}u=\sup_{\partial M}u.
\]
\end{theorem}

We refer the reader to \cite[Theorem 0.9]{IPS-Crelle} for a detailed proof of the
previous result in the unweighted setting. Although the proof of this theorem can be deduced
adapting to the weighted laplacian $\Delta_{f}$ the arguments in \cite{IPS-Crelle} (indeed, all results obtained there can be adapted to the weighted case with only minor modifications of the proofs),
we provide here a shorter and more elegant argument. In order to do this we will need the following preliminary fact that will be proved without the use of any capacitary argument and, therefore, can be adapted to deal also with the $f$-parabolicity under Dirichlet boundary conditions; see \cite{PPS-L1Liouville} for old results and recent advances on the Dirichlet parabolicity of (unweighted) Riemannian manifolds.

\begin{proposition}\label{prop_equivfpar}
Let $M_f$ be a weighted manifold with boundary $\partial M$ oriented by the exterior unit normal $\nu$ and consider the warped product manifold $\hat{M}=M\times_{e^{-f}}\mathbb{T}$, with boundary $\partial\hat{M}=\partial{M}\times\mathbb{T}$ oriented by the exterior unit normal $\hat \nu(x,\theta) = (\nu(x),0)$.  Here we are setting $\mathbb{T}=\R /\mathbb{Z}$, normalized so that $\mathrm{vol}(\mathbb{T})=1$. Then $M_f$ is $f$-parabolic if and only if $\hat M$ is parabolic.
\end{proposition}
\begin{proof}
To illustrate the argument we are going use pointwise computations for $C^{2}$ functions which however can be easily formulated in weak sense for functions in $C^{0}\cap W^{1,2}_{loc}$. We also recall that the $f$-Laplacian $\Delta_{f}$ on $M$ is related to the Laplace-Beltrami operator $\hat\Delta$ of $\hat M$ by the formula
\[
\hat\Delta = \Delta_f + f^{-2} \Delta_{\mathbb{T}}.
\]
With this preparation, assume first that $\hat{M}$ is parabolic. We have to show that any (regular enough) solution of the problem
\[
\left\{
\begin{array}
[c]{ll}%
{\Delta}_{f}  u\geq0 & \text{on }\mathrm{int}{M}\\
\dfrac{\partial  u}{\partial {\nu}}\leq0 & \text{on }\partial {M}\\
\sup_{{M}} u<+\infty &
\end{array}
\right.
\]
must be constant. To this end, having selected $u$ we simply define $\hat u(x,t)=u(x)$, $(x,t)\in\hat{M}$, and we observe that $\hat u$ satisfies the analogous problem on $\hat M$. Since $\hat M$ is parabolic, $\hat u$ and hence $u$ must be constant, as required.

Conversely, assume that $M_f$ is $f$-parabolic and let $u$ be a (regular enough) solution of the problem
\begin{equation}
\label{subneumannproblem2}
\begin{array}
[c]{ccc}%
\left\{
\begin{array}
[c]{ll}%
\hat \Delta u\geq0 & \text{on }\mathrm{int} \hat{M}\\
\dfrac{\partial u}{\partial \hat\nu}\leq0 & \text{on }\partial \hat{M}=\partial M\times \mathbb{T}\\
\sup_{\hat M}u<+\infty. &
\end{array}
\right.
\end{array}
\end{equation}
By  translating and scaling we may assume that $\sup u=1 $ and since $\max\{u,0\}$ is again a solution of  \eqref{subneumannproblem2}, we may in fact assume that $0\leq u\leq 1$. Let
\[
\bar{u}(x) = \int_{\mathbb{T}} u(x,t)dt.
\]
Recalling that rotations in $\mathbb{T}$ are isometries of $\hat{M}$, and therefore commute with the Laplacian $\hat \Delta$, we deduce that
\[
\Delta_{f}  \bar u (x) = \hat \Delta \bar u (x)= \int_{\mathbb{T}} \hat \Delta u (x,t) dt \geq 0
\]
and
\[
\frac{\partial \bar u}{\partial \nu}\leq 0,
\]
so, by the assumed $f$-parabolicity of $M_f$, $\bar u$ is constant, and
\[
0=\Delta_f \bar{u}=\int_{\mathbb{T}} \hat \Delta u (x,t).
\]
Since $\hat \Delta u \geq 0$ we conclude that $\hat \Delta u=0$ in $\mathrm{int} \,\hat{M}$ . Applying the above argument to $u^2$, which is again a solution of \eqref{subneumannproblem2}, we obtain that
\[
0=\hat \Delta u^2 = 2 u \hat \Delta u + 2 |\hat \nabla u|^2 =  2 |\hat \nabla u|^2 \quad \text{ in } \mathrm{int} \, \hat{M},
\]
and therefore $u$ is constant, as required.
\end{proof}
\medskip

We are now ready to present the

\begin{proof}[Proof of Theorem \ref{thm_fAhlfors}]
Assume first that $M_f$ is $f$-parabolic. As a consequence of Proposition \ref{prop_equivfpar} this is equivalent to the parabolicity of $\hat{M}$ which, in turns, is equivalent to the validity of the following Ahlfors-type maximum principle (see \cite[Theorem 0.9]{IPS-Crelle}). For every domain $\hat{D}\subseteq \hat{M}$
with $\partial_{0}\hat{D}\neq\emptyset$ and for every $u\in C^{0}(\overline{\hat{D}})  \cap W_{loc}^{1,2}(  \hat{D})  $ satisfying
\begin{equation}\label{eq_ahlforsmhat}
\left\{
\begin{array}
[c]{ll}%
\hat{\Delta} u\geq0 & \text{on }\mathrm{int}\hat{D}\\
\dfrac{\partial u}{\partial\hat{\nu}}\leq0 & \text{on }\partial_{1}\hat{D}\\
\sup\limits_{\hat{D}}u<+\infty &
\end{array}
\right.
\end{equation}
in the weak sense, it holds%
\[
\sup_{\hat{D}}u=\sup_{\partial_{0}\hat{D}}u.
\]
Furthermore, in case $\hat{D}=\hat{M}$, for every $u\in C^{0}(\overline{\hat{D}})  \cap W_{loc}^{1,2}(  \hat{D})  $ satisfying
\[
\left\{
\begin{array}
[c]{ll}%
\hat{\Delta} u\geq0 & \text{on }\mathrm{int}\hat{M}\\
\sup\limits_{\hat{M}}u<+\infty &
\end{array}
\right.
\]
in the weak sense, it holds%
\[
\sup_{\hat{M}}u=\sup_{\partial_{0}\hat{M}}u.
\]
In particular, if $D\subset M$ is any smooth domain with $\partial_0 D\neq\emptyset$ and if $u\in C^{0}\left(
\overline{D}\right)  \cap W_{loc}^{1,2}\left(D\right)  $ satisfies
\[
\left\{
\begin{array}
[c]{ll}%
\Delta_f u\geq0 & \text{on }\mathrm{int}D\\
\dfrac{\partial u}{\partial\nu}\leq0 & \text{on }\partial_{1}D\\
\sup\limits_{D}u<+\infty &
\end{array}
\right.
\]
in the weak sense, then $\hat{u}(x,t)=u(x)$ is a solution of \eqref{eq_ahlforsmhat} on $\hat{D}\times\mathbb{T}$. Hence, the parabolicity $\hat{M}$ in the form of the Ahlfors-type maximum principle implies that
\[
\sup_Du=\sup_{\hat{D}}\hat{u}=\sup_{\partial_{0}\hat{D}=\partial_0D\times\mathbb{T}}\hat{u}=\sup_{\partial_0D}u.
\]
The same reasoning applies in case $D=M$.

Conversely, assume that for every domain $D\subseteq M$
with $\partial_{0}D\neq\emptyset$ and for every $u\in C^{0}\left(
\bar{D}\right)  \cap W_{loc}^{1,2}(  D_f)  $ satisfying
\[
\left\{
\begin{array}
[c]{ll}%
\Delta_f u\geq0 & \text{on }\mathrm{int}D\\
\dfrac{\partial u}{\partial\nu}\leq0 & \text{on }\partial_{1}D\\
\sup\limits_{D}u<+\infty &
\end{array}
\right.
\]
in the weak sense, it holds%
\[
\sup_{D}u=\sup_{\partial_{0}D}u.
\]
Suppose by contradiction that $M_f$ is not $f$-parabolic. Then there exists a non-constant function $v\in C^{0}\left(
M\right)  \cap W_{loc}^{1,2}\left(  M_f\right)  $ satisfying
\[
\left\{
\begin{array}
[c]{ll}%
\Delta_f v\geq0 & \text{on }\mathrm{int}M\\
\dfrac{\partial v}{\partial\nu}\leq0 & \text{on }\partial_{1}M\\
\sup\limits_{M}v<+\infty &
\end{array}
\right.
\]
in the weak sense. Given $\eta<\sup_M v$ consider the domain $D_{\eta}=\{x\in M:v(x)>\eta
\}\neq\emptyset$. We can choose $\eta$ sufficiently close to $\sup_M v$ in
such a way that $\mathrm{int}M\not \subseteq D_{\eta}$. In particular,
$\partial D_{\eta}\subseteq\left\{  v=\eta\right\}  $ and $\partial
_{0}D_{\eta}\neq\emptyset$. Now, $v\in C^{0}\left(  \overline{D
}_{\eta}\right)  \cap W_{loc}^{1,2}\left(  (D_{\eta})_f\right)  $ is a
bounded above weak Neumann subsolution of the $f$-Laplacian equation on $D_{\eta}$.
Moreover,
\[
\sup_{\partial_{0}D_{\eta}}{v}=\eta<\sup_{D_{\eta}}{v},
\]
contradicting our assumptions.
\end{proof}

From the geometric point of view, it can be proved that $f$--parabolicity is related to the growth rate of
the weighted volume of intrinsic metric objects. Indeed, exploiting a result due to A. Grigor'yan, \cite{Gr1}, one can prove the following result; see also Theorem 0.7 and Remark 0.8 in \cite{IPS-Crelle}). For the sake of clarity, we recall from the Introduction that the weighted volume of the metric ball $B^{M}_{R}(o) = \{ x\in M : \mathrm{dist}_{M}(x,o)<R\}$ of $M_{f}$ is defined by $\vol_{f}(B^{M}_{R}\!(o)) = \int_{B^{M}_{R}\!(o)} \dd M_{f}$.
\begin{proposition}\label{prop-fparab-volume}
Let $M_f$ be a complete weighted manifold
with boundary $\partial M\neq\emptyset$. If, for some reference point $o\in
M$,
\[
\mathrm{vol}_f (B_{R}^{M}\!(o) )
=\CO(R^2), \quad \mathrm{as}\quad R\rightarrow+\infty.
\]
then $M_f$ is $f$-parabolic.
\end{proposition}
\begin{proof}
Set $\Omega_R:=B_{R}^{M}(o)\times \mathbb{T}$. Then, as a consequence of Fubini's Theorem,
\[
\vol_f(B_{R}^{M} \!(o))=\int_{B_{R}^{M}\!(o)}\dd M_f=\int_{\Omega_R}\dd \hat{M}=\vol(\Omega_R).
\]
Denote by $B_R^{\hat{M}}\!(\hat{o})$ the geodesic ball in $\hat{M}$ with reference point $\hat{o}=(o,\hat{t})$. Given an arbitrary point $\hat{x}=(x,t)\in\hat{M}$, let $\hat{\alpha}=(\alpha,\beta):[0,1]\rightarrow \hat{M}$ be a curve in $\hat{M}$ such that $\hat{\alpha}(0)=\hat{o}$ and $\hat{\alpha}(1)=\hat{x}$. Then
\begin{align*}
\ell(\hat{\alpha})=&\int_{0}^{1}\|\hat{\alpha}^{\prime}(s)\|\dd s\\
=&\int_{0}^{1}\sqrt{({\alpha}^{\prime}(s))^2+e^{-2 f(\alpha(s))}({\beta}^{\prime}(s))^2}\dd s\\
\geq &\int_{0}^{1}\|\alpha^{\prime}(s)\|\dd s\\
\geq & \textrm{dist}_M(x,o).
\end{align*}
Hence, as a consequence of the previous chain of inequalities, if $\hat{x}\in B_R^{\hat{M}}\!(\hat{o})$, then $x\in B_R^{M}\!(o)$, which in turns implies that
\[
B_R^{\hat{M}}\!(\hat{o})\subseteq \Omega_R.
\]
In particular,
\[
\vol(B_R^{\hat{M}}\!(\hat{o}))\leq \vol(\Omega_R)\leq C R^2
\]
for $R$ sufficiently large. The conclusion now follows  from the above mentioned result by Grigor'yan  and from the fact that, as proven in Proposition \ref{prop_equivfpar}, the parabolicity of $\hat{M}$ is equivalent to the $f$-parabolicity of $M_f$.
\end{proof}

\section{The height estimates}\label{HeightEst}
In this section we prove the main result of the paper, Theorem \ref{th_fheightestimate}, and show how the boundedness assumption can be dropped in dimension $2$; Corollary \ref{coro_fheightestimate}.
To this end, we need some preparation: first we derive some basic formulas concerning the weighted Laplacian of the height function and the angle function of the Killing graph; see Proposition \ref{prop-formulas}. Next we extend to the weighted setting and for Killing graphs a crucial volume estimate obtained in \cite{IPS-Crelle, LW}; see Lemma \ref{lemma_volume}. This estimate will allow us to use the global maximum principles introduced in Section \ref{PotTheory}. The proofs of Theorem \ref{th_fheightestimate} and of its Corollary \ref{coro_fheightestimate} will be provided in Section \ref{subsection-proof}.

\subsection{Some basic formulas}
This section aims to prove the following
\begin{proposition}\label{prop-formulas}
Let $( M, \g_{M})$ be a complete $(n+1)$-dimensional Riemannian manifold endowed with a complete Killing vector field $Y$ whose orthogonal
distribution has constant rank $n$ and is integrable. Let $(P,\g_{P})  $ be an integral leaf of that distribution  and let $\Sigma=\mathrm{Graph_{\Omega}}(u)$ be a Killing graph over a smooth domain $\Omega\subset P$, with upward unit normal $\CN$. Set $\psi=-\log|Y|$. Then, for any constant $C \in \R$ the following equations hold on $(\Sigma, \g)$:
\begin{eqnarray}
\label{eq-Deltafu}
\Delta^{\Sigma}_{C\psi}u & = & \left( nH+(C-2)\langle {\nabla}^{M} \psi, \CN \rangle_{M} \right)
e^{2\psi}\langle Y,\CN \rangle_{M},\\
\label{eq-DeltafAngle}
\Delta^{\Sigma}_{C\psi}\langle Y, \CN \rangle_{M} & = & -\langle Y, \CN\rangle_{M}
\left( |A|^2+\Ric^{M}_{C\psi}(\CN,\CN) \right)- nY^T(H_{C\psi}),
\end{eqnarray}
where $\Ric^{M}_{C\psi}$ is the Bakry-\'Emery Ricci tensor of the weighted manifold $M_{C\psi}$, $|A|$ is the norm of the second fundamental form of $\Sigma$ and $H$, $H_{C\psi}$ denote respectively the mean curvature of $\Sigma$ and its $C\psi$-weighted modified version.
\end{proposition}
\begin{proof}
Observe that
\[
u=s|_{\Sigma},
\]
where, we recall, $s$ is the flow parameter of $Y$. Using ${\nabla}s=e^{2\psi} Y$, we have
\[
\nabla^{\Sigma} u = \nabla^{\Sigma} s = e^{2\psi} Y^T.
\]
Thus, letting $\{e_i\}$ be an orthonormal basis of $T\Sigma$, and recalling that,
since $Y$ is Killing, $\langle \nabla^{M}_V Y, V\rangle_{M} = 0$ for every vector $V$, we compute
\begin{align*}
&  \Delta^{\Sigma}u=\sum_{i=1}^{n}\langle\nabla_{e_{i}}^{\Sigma}\nabla^{\Sigma}u,e_{i}\rangle\\
&\,\, = \sum_{i=1}^{n}\langle\nabla^{\Sigma}_{e_{i}} e^{2\psi}Y^{T},e_{i}\rangle\\
&\,\, = e^{2\psi}\sum_{i=1}^{n}\langle{\nabla}^{M}_{e_{i}}(Y-\langle Y,\CN \rangle_{M}
\CN),e_{i}\rangle_{M} +2\sum_{i=1}^{n}\langle e_{i},{\nabla^{\Sigma}}\psi\rangle\langle e_{i},{\nabla^{\Sigma}}s\rangle\\
&= e^{2\psi}\sum_{i=1}^{n}\langle{\nabla}^{M}_{e_{i}}Y,e_{i}\rangle_{M}
-\langle Y,\CN\rangle_{M} e^{2\psi}\sum_{i=1}^{n}\langle{\nabla}^{M}_{e_{i}} \CN,e_{i}\rangle_{M}
+2\sum_{i=1}^{n}\langle e_{i},\nabla^{\Sigma}\psi
\rangle\langle e_{i},\nabla^{\Sigma}u\rangle\\
&  \,\,=nHe^{2\psi}\langle Y,\CN \rangle_{M}+2\langle\nabla^{\Sigma}\psi,\nabla^{\Sigma}u\rangle.
\end{align*}
Equation \eqref{eq-Deltafu} follows since, by definition,
\[
\Delta_{C\psi}^{\Sigma}u= \Delta^\Sigma u- C\langle\nabla^\Sigma \psi, \nabla^\Sigma u\rangle.
\]

As for equation \eqref{eq-DeltafAngle}, note that, since $\CN$ is a unit normal and   $Y$ is a Killing vector field $\nabla^{M}_\CN Y$ is tangent to $\Sigma$  and,  for every vector $X$ tangent to $\Sigma$,
\[
\langle \nabla^{M}_X \CN, Y\rangle_{M} = \langle \nabla^{M}_X \CN, Y^T + \langle Y, \CN \rangle_{M} \CN \rangle_{M} =
\langle \nabla^{M}_X \CN, Y^T\rangle_{M} = -\langle A_{\CN}\,Y^T, X\rangle,
\]
so that
\begin{equation*}
\langle \nabla^\Sigma \langle Y,\CN \rangle_{M}, X\rangle
=
X \langle Y,\CN\rangle_{M} =
-\langle \nabla^{M}_\CN Y, X\rangle_{M} + \langle Y,  \nabla^{M}_{X} \CN\rangle_{M}
= - \langle \nabla^{M}_\CN Y + A_{\CN} \, Y^T, X\rangle_{M}
\end{equation*}
and therefore
\[
\nabla^{\Sigma}\langle Y, \CN \rangle_{M}=-{\nabla}^{M}_\CN Y- A_{\CN} Y^T.
\]
Moreover, using the Codazzi equations, it is not difficult to prove that
\[
\Delta^{\Sigma}\langle Y,\CN\rangle_{M}=-\bigl(|A|^{2}+{\Ric^{M}}(\CN,\CN)\bigr)\langle
Y, \CN \rangle_{M}-nY^T(H),
\]
see, e.g.,
\cite[Prop. 1]{Fornari-Ripoll-Illinois}.
 Using the definition of
$H_{C\psi}= H+\frac 1n \langle \nabla^{M}( C\psi), \CN\rangle_{M}$  we also have
\begin{align*}
-nY^T(H)
&=-nY^T(H_{C\psi})+CY^T\langle {\nabla}^{M} {\psi}, \CN\rangle_{M}\\
&=-nY^T(H_{C\psi})+C\langle{\nabla}^{M}_Y {\nabla}^{M} {\psi}, \CN \rangle_{M}
- C\langle Y,\CN\rangle_{M} {\Hess^{M}}(\psi)(\CN,\CN)+
C\langle{\nabla}^{M} {\psi},{\nabla}^{M}_{Y^T} \CN\rangle_{M}\\
&=-nY^T(H_{C\psi})+C\langle{\nabla}^{M}_{{\nabla} {\psi}} Y, \CN \rangle_{M}
-C\langle Y,\CN\rangle_{M} {\Hess^{M}}(\psi)(\CN,\CN)
+ C\langle{\nabla}^\Sigma {\psi},({\nabla}^{M}_{Y^T} \CN)^{T}\rangle\\
&=-nY^T(H_{C\psi})+C\langle\nabla^\Sigma {\psi},\nabla^\Sigma \langle Y, \CN \rangle_{M} \rangle
-C\langle Y,\CN\rangle_{M} {\Hess}^{M}(\psi)(\CN,\CN),
\end{align*}
where we have used that:\\
\begin{itemize}
 \item [-]  $\nabla^{M}_Y \nabla^{M}\psi = \nabla^{M}_{\nabla^{M} \psi}Y,$  since $\psi$ depends only on the $P$-variables and $Y=\partial_{s}$;\smallskip
 \item [-] $\langle \nabla^{M}_{\nabla^{M} \psi} Y, \CN \rangle_{M} = -\langle \nabla^{M}_\CN Y,\nabla^{M} \psi\rangle_{M}$ and $\langle \nabla^{M}_\CN Y, \CN \rangle_{M} =0$, since $Y$ is Killing;\smallskip;
  \item [-] $\langle\nabla^{M} \psi, \nabla^{M}_{Y^T} \CN \rangle_{M} = \langle \nabla^\Sigma \psi, (\nabla^{M}_{Y^T}\CN)^{T} \rangle$, since $\langle \CN, \nabla^{M}_{Y^T}\CN \rangle_{M}
 = \frac 12 Y^T\langle \CN, \CN \rangle_{M} = 0$;\smallskip
 \item [-] the identity
\begin{equation*}
\begin{split}
\nabla^\Sigma \psi\langle Y, \CN \rangle_{M}
&= \langle \nabla^{M}_{\nabla^\Sigma \psi} Y, \CN \rangle_{M}
+ \langle Y, \nabla^{M}_{\nabla^\Sigma \psi} \CN \rangle_{M}\\
&= \langle \nabla^{M}_{\nabla^{M} \psi} Y, \CN \rangle_{M} -
\langle \CN , \nabla^{M} \psi \rangle_{M} \langle \nabla^{M}_\CN Y, \CN\rangle_{M}\\
&\,\,\,\,\,\, +\langle Y^T , (\nabla_{\nabla^\Sigma \psi} \CN)^{T}\rangle
+ \langle \CN, Y \rangle_{M} \langle \CN, \nabla^{M}_{\nabla^\Sigma \psi} \CN \rangle_{M}\\
&= \langle \nabla^{M}_{\nabla^{M} \psi} Y, \CN \rangle_{M}
+\langle Y^T ,(\nabla^{M}_{\nabla^\Sigma \psi} \CN)^{T} \rangle\\
&=\langle \nabla^{M}_{\nabla^{M} \psi} Y, \CN \rangle_{M}
+\langle \nabla^\Sigma \psi ,(\nabla^{M}_{Y^T} \CN)^{T} \rangle.
\end{split}
\end{equation*}

\end{itemize}
Inserting the above identities into
\[
\Delta^\Sigma_{C\psi}\langle Y, \CN \rangle_{M} = \Delta^\Sigma \langle Y,\CN \rangle_{M} - C\langle \nabla^\Sigma \psi, \nabla^\Sigma \langle Y, \CN\rangle\rangle
\]
and recalling that
\[
\Ric^{M}_{C\psi}= \Ric^{M}+C\, \Hess^{M}(\psi)
\]
yield the validity of \eqref{eq-DeltafAngle}.
\end{proof}

\subsection{Weighted volume estimates}

In this section we extend to the weighted setting of Killing graphs, with
prescribed (weighted) mean curvature, an estimate of the extrinsic volume
originally obtained in \cite{LW} and later extended in \cite{IPS-Crelle}.

\begin{lemma} \label{lemma_volume}
Let $( M, \g_{M})$ be a complete $(n+1)$-dimensional Riemannian manifold endowed with a complete Killing vector field $Y$ whose orthogonal distribution has constant rank $n$ and is integrable.
Let $(P,\g_{P})  $ be an integral leaf of that distribution so that $M$ can be identified with
$P \times_{e^{-\psi}}\R$, where $\psi=-\log|Y|$.
Let $\Sigma=\mathrm{Graph_{ \Omega}}(u)$ be a Killing graph over a smooth domain $\Omega\subset P$, with mean curvature $H$ with respect to the upward unit normal $\CN$ and let $\pi:\Sigma\to P$ be the projection map. Then, for any $y_{0}=(  x_{0},u(x_0))  \in \Sigma$ and for every $R>0$,
\[
\pi(B_R^\Sigma(y_0))
\subseteq
\Omega_{R}\left(x_{0}\right)
\]
where we have set
\[
\Omega_{R}\left(  x_{0}\right)  =B_{R}^{P}(x_{0})\cap\Omega.
\]
Moreover, assume that given $D\in \R$,
\begin{equation}
A:=\sup_{\Omega}\left\vert u (x)  e^{-\psi (x)
}\right\vert +\sup_{\Omega}\left\vert H_{D\psi} (x)  \right\vert
<+\infty,\label{volume-u+Hf}.
\end{equation}
Then, there exists a constant $C>0$,
depending on $n$ and $A$, such that, for every $\delta,R>0$, the corresponding $D\psi$-volume of the intrinsic ball of $\Sigma$ satisfies
\begin{equation}\label{eq_volume}
\vol_{D\psi}B_{R}^{\Sigma}\left(  y_{0}\right)  \leq C\left(1+\frac{1}{\delta R}\right)  \vol_{D\psi}\left(  \Omega_{\left( 1+\delta\right)  R}\left(  x_{0}\right)  \right),
\end{equation}
where $y_{0}=(x_{0},u(x_{0}))\in\Sigma$ is a reference origin.
\end{lemma}

\begin{proof}
The Riemannian metric of $M$ writes as $\g_{M}=\g_{P}+e^{-2\psi} \dd s \otimes \dd s$.
Let $y_{0}=(x_{0},u(x_{0}))$ and $y=(x,u(x))$ be points in $\Sigma$ connected by the curve $(\alpha(t), u(\alpha(t))$, where $\alpha(t)$, $t\in\lbrack0,1]$, is an
arbitrary path connecting $x_{0}$ and $x$ in $\Omega\subseteq P$. Writing
$s(t)=u(\alpha(t))$
we have
\begin{align*}
\int_{0}^{1}\left\{|\alpha^{\prime}\left(  t\right)  ^{2}+e^{-2\psi(\alpha(t))}s^{\prime}(t)^{2}\right\}^{\frac{1}{2}} dt &
\geq \int_{0}^{1}|\alpha^{\prime}(t)|dt \geq
d_{P}(x_{0},x).
\end{align*}
Thus, if $y\in B_{R}^{\Sigma}(y_{0})$ we deduce that $x\in\Omega
_{R}\left(  x_{0}\right)  $, proving the first half of the lemma.

Now we compute the volume of $B_{R}^{\Sigma}(y_{0})$.  Since $\pi(B_{R}^{\Sigma}(y_{0}))\subset\Omega_{R}\left(  x_{0}\right)  $ we have
\begin{align}
\mathrm{vol}_{D\psi}(B_{R}^{\Sigma}(y_{0})) &  =\int_{\pi(B_{R}^{\Sigma}(y_{0}))}\sqrt{e^{2\psi}+|\nabla^{P}u|^{2}}e^{-(D+1)\psi}\,dP\label{volume-1}\\
&  \leq\int_{\Omega_{R}\left(  x_{0}\right)  }
\sqrt{e^{2\psi}+|\nabla^{P}u|^{2}}e^{-(D+1)\psi}dP\nonumber\\
&  =\int_{\Omega_{R}\left(  x_{0}\right)  }
\frac{|\nabla^{P}u|^{2}}{\sqrt{e^{2\psi}+|\nabla^{P}u|^{2}}}e^{-(D+1)\psi}\,dP\nonumber\\
&  +\int_{\Omega_{R}\left(  x_{0}\right)  }\frac
{e^{2\psi}}{\sqrt{e^{2\psi(x)}+|\nabla^{P}u|^{2}}}e^{-(D+1)\psi}dP.\nonumber
\end{align}
We then consider the vector field
\[
Z=\rho u\frac{e^{-\left(D+1\right)  \psi}\nabla^{P}u}{\sqrt{e^{2\psi}
+|\nabla^{P}u|^{2}}},
\]
where the function $\rho$ is given by
\[
\rho(x)=
\begin{cases}
1 & \mathrm{on}\quad B_{R}(x_{0})\\
\frac{\left(  1+\delta\right)  R-r(x)}{\delta R} & \mathrm{on}\quad B_{\left(
1+\delta\right)  R}(x_{0})\backslash B_{R}(x_{0})\\
0 & \mathrm{elsewhere},
\end{cases}
\]
with $r(x) = \mathrm{dist}_{P}(x,x_{0})$.
Recalling equation \eqref{capillary-2} in the Introduction, i.e.,
\[
\div^{P}\left\{  \frac{\nabla^{P}u}{\sqrt{e^{2\psi}+\left\vert
\nabla^{P}u\right\vert ^{2}}}\right\}=nH+\frac{\left\langle
\nabla^{P}u,\nabla^{P}\psi\right\rangle }{\sqrt{e^{2\psi}+\left\vert \nabla
^{P}u\right\vert ^{2}}}
\]
we compute
\begin{align*}
\div Z  & =e^{-\left(  D+1\right)  \psi}\left\{  u\frac
{\left\langle \nabla^{P}\rho,\nabla^{P}u\right\rangle }{\sqrt{e^{2\psi}+\left\vert \nabla
^{P}u\right\vert ^{2}}}+\rho\frac{\left\vert \nabla
^{P}u\right\vert ^{2}}{\sqrt{e^{2\psi}+\left\vert \nabla
^{P}u\right\vert ^{2}}
}+n\rho uH_{D\psi}\right\}  \\
& =e^{-D\psi}\left\{  u e^{-\psi}\frac{\left\langle \nabla^{P}
\rho,\nabla^{P}u\right\rangle }{\sqrt{e^{2\psi}+\left\vert \nabla
^{P}u\right\vert ^{2}}}+\rho e^{-\psi}\frac{\left\vert \nabla^{P}u\right\vert ^{2}
}{\sqrt{e^{2\psi}+\left\vert \nabla
^{P}u\right\vert ^{2}}}+n\rho u e^{-\psi}H_{D\psi}\right\}  \\
& \geq e^{-D\psi}\left\{  -\left\vert u e^{-\psi}\right\vert
\left\vert \nabla^{P}\rho\right\vert +\rho e^{-\psi}
\frac{\left\vert \nabla^{P}u\right\vert ^{2}}{\sqrt{e^{2\psi}+\left\vert \nabla
^{P}u\right\vert ^{2}}}-n\rho\left\vert u e^{-\psi}\right\vert
\left\vert H_{D\psi}\right\vert \right\}  .
\end{align*}
Since $Z$ has compact support in $\Omega_{\left(  1+\delta\right)  R}$,
applying the divergence theorem and using the properties of $\rho$, from the
above inequality we obtain
\begin{align*}
\int_{\Omega_{R}\left(  x_{0}\right)  }\frac{\left\vert
\nabla^{P}u\right\vert ^{2}}{\sqrt{e^{2\psi}+\left\vert \nabla
^{P}u\right\vert ^{2}}}e^{-\psi}e^{-D\psi}\text{ }dP  & \leq\frac{1}{\delta R}\int_{\Omega_{\left(
1+\delta\right)  R}\left(  x_{0}\right)  }\left\vert u e^{-\psi}\right\vert
e^{-D\psi}\text{ }dP\\
& +n\int_{\Omega_{\left(  1+\delta\right)  R}\left(  x_{0}\right)  }
\left\vert u e^{-\psi}\right\vert \left\vert H_{D\psi}\right\vert
e^{-D\psi}\text{ }dP.
\end{align*}
Inserting this latter into (\ref{volume-1}) we get
\begin{align*}
\mathrm{vol}_{D\psi}(B_{R}^{\Sigma}(y_{0})) &  \leq\frac{1}{\delta R}
\int_{\Omega_{\left(  1+\delta\right)  R}\left(  x_{0}\right)  }
\left\vert u e^{-\psi}\right\vert e^{-D\psi}\text{ }dP\\
&  +n\int_{\Omega_{\left(  1+\delta\right)  R}\left(  x_{0}\right)  }
\left\vert u e^{-\psi}\right\vert \left\vert H_{D\psi}\right\vert
e^{-D\psi}\text{ }dP.
\end{align*}
To conclude the desired volume estimate, we now recall that, by assumption,
\[
\sup_{\Omega}\left\vert u e^{-\psi}\right\vert +\sup_{\Omega
}\left\vert H_{D\psi}\right\vert <+\infty.
\]
The proof of the Lemma is completed.
\end{proof}
\begin{remark}\label{rmk-fvolumes}
\rm{
We note for further use that if, in the previous Lemma, we assume that $\inf_P\psi>-\infty$, then the following more general inequality holds:
\[
\vol_{C\psi} B_{R}^{\Sigma} (y_0)  \leq A \vol_{D\psi}(\Omega_R(x_0))  ,
\]
for any constant $C>D$ and any $R \gg 1$.
}
\end{remark}

\subsection{Proofs of Theorem \ref{th_fheightestimate} and Corollary \ref{coro_fheightestimate}}\label{subsection-proof}

We are now in the position to give the
\begin{proof}[Proof (of Theorem \ref{th_fheightestimate})]
Since $H_{\psi} \equiv \mathrm{const}$, it follows by equation \eqref{eq-Deltafu} that
\begin{align*}
\Delta^{\Sigma}_{\psi}(H_{\psi} u)
&= nH_{\psi}He^{2\psi}\langle Y,\CN \rangle_{M}-
H_{\psi}e^{2\psi}\langle \nabla^{M} {\psi}, \CN \rangle_{M} \langle Y, \CN \rangle_{M}\\
&= nH^2e^{2\psi}\langle Y, \CN \rangle_{M} + He^{2\psi} \langle \nabla^{M} {\psi}, \CN \rangle_{M} \langle Y, \CN \rangle_{M}\\
&\,\,\,\,\, -He^{2\psi}\langle \nabla^{M} {\psi}, \CN \rangle_{M} \langle Y,\CN \rangle_{M}
-\frac1n e^{2\psi} \langle \nabla^{M} {\psi}, \CN \rangle_{M}^{2} \langle Y, \CN \rangle_{M}\\
&\leq nH^2e^{2\psi}\langle Y, \CN \rangle_{M}.
\end{align*}
Combining this inequality with equation \eqref{eq-DeltafAngle}, it is straightforward to prove that, under our assumptions on $\Ric^{M}_{\psi}$ and $H_{\psi}$, the function
\[
\varphi (x) =H_{\psi}u(x) e^{-2\sup_\Omega \psi}+\langle Y_{x}, \CN_{x} \rangle_{M}
\]
satisfies
\[
\Delta^{\Sigma}_{ \psi} \varphi  \leq0\text{ on }\Sigma.
\]
On the other hand, using assumptions (b) and (e) we can apply Lemma \ref{lemma_volume} to deduce that
\[
\vol_{\psi}(B_{R}^{\Sigma})=\mathcal{O}(R^2)  \text{, as }
R\rightarrow+\infty.
\]
In particular, by Proposition \ref{prop-fparab-volume}, $\Sigma$ is parabolic with respect to the weighted Laplacian $\Delta^{\Sigma}_{\psi}$.
Since, again by (e), $\varphi$ is a bounded function and, according to (d), $u\equiv0$ on
$\partial\Omega$, an application of the Ahlfors maximum principle stated in Theorem \ref{thm_fAhlfors} gives
\[
\inf_{\Omega}\left( H_{\psi}u e^{-2\sup_{\Omega} \psi}+\langle Y, \CN \rangle_{M}\right)
=\inf_{\partial\Omega}  \langle Y, \CN \rangle_{M} \geq0.
\]
Combining this latter with the fact that $\langle Y, \CN \rangle_{M} \leq e^{-\psi}$ we get
the desired upper estimate on $u$.

Finally, note that equation \eqref{eq-Deltafu} can also be written in the form:
\[
\Delta^{\Sigma}_{3\psi} u=H_{\psi} e^{2\psi}\langle Y, \CN \rangle_{M}.
\]
Since, according to Remark \ref{rmk-fvolumes}, $\Sigma$ is also parabolic with respect to the weighted laplacian $\Delta^{\Sigma}_{3\psi}$ and $u$ is a bounded $3\psi$-superharmonic function,
the desired lower estimate on $u$ follows again as an application of Theorem \ref{thm_fAhlfors}.
\end{proof}

\begin{remark} \label{RmkHLR}
{
\rm  It is clear from the proof that if we consider, as in \cite{HLR-TAMS}, the function $\varphi=cH_\psi e^{-2\psi} u +\langle Y,\CN\rangle_M$, for some $0<c\leq 1$, it is possible to extend the height estimate to the case where
$\Ric_\psi^M\geq -G^2(x)$ in a neighborhood of $\Omega\times \R$, provided $(1-c)H^2(x)/n\geq G^2(x)$. The resulting estimate becomes
\[
0\leq ue^{-\psi}\leq \frac{e^{2(\sup_\Omega \psi-\inf_\Omega\psi)}}{c|H_\psi|}.
\]
}
\end{remark}

To conclude this section, we show how the boundedness assumption (e) can be dropped in the case of $2$-dimensional graphs.

\begin{proof}[Proof (of Corollary \ref{coro_fheightestimate})]
Let $\Sigma$ be a $2$-dimensional Killing graph over a domain $\Omega\subset P$ of constant weighted mean curvature $H_\psi<0$. Recall that the {\it Perelman scalar curvature} of $R_\psi$ of $M=P\times_{e^{-\psi}}\R $ is defined by
\[
R^{M}_\psi=R+2 \Delta \psi-|\nabla \psi|^2=\mathrm{Tr}(\Ric^{M}_\psi)+\Delta_\psi\psi.
\]
Assume that $R^{M}_{\psi} \geq 0$. Then, as a consequence of a result due to Espinar \cite[Theorem 4.2]{Esp}, for every $y=(x,u(x))\in\Sigma$ it holds
\[
|u(x)e^{-\psi}|\leq\mathrm{dist}(y,\partial \Sigma)\leq C(|H_\psi|)< +\infty.
\]
On the other hand, a straightforward calculation shows that
\[
\Delta_\psi \psi=e^{2\psi}\Ric^{M}_\psi(Y,Y).
\]
Hence, $R^{M}_{\psi} \geq 0$ provided $\Ric^{M}_{\psi} \geq 0$. Putting everything together, it follows that  condition (e) in Theorem \ref{th_fheightestimate} is automatically satisfied.
\end{proof}


\section{Height estimates in model manifolds}\label{section-examples}

In this section we construct rotationally symmetric examples of  Killing graphs with constant
weighted mean curvature and  exhibit explicit estimates on the maximum of
their weighted height in terms of the weighted mean curvature. When the base space is
$\R ^{n}$ and the Killing vector field has constant length $1$ (hence
the ambient space is $\R ^{n}\times\R $) these graphs are
standard half-spheres and the estimate on their maximal height is precisely
the reciprocal of the mean curvature; \cite{He, Se}.


We shall assume that the induced metric $\g_{P}$ on $P$ is rotationally
invariant. More precisely, we suppose that $P$ is a model space with pole at
$o$ and Gaussian coordinates $(r,\theta)\in\left(  0,R\right)  \times
\mathbb{S}^{n-1}$, $R\in(0,+\infty],$ in terms of which $g_{P}$ is expressed
by
\[
g_{P}=dr^{2}+\xi^{2}(r)d\theta^{2},
\]
for some $\xi\in C^{\infty}([0,R))$ satisfying 
\[
\begin{cases}
& \xi>0  \text{ on }\left(  0,R\right) \\
& \xi^{(2k)}\left(  0\right)  =0, \,\, k\in \mathbb{N}  \\
& \xi^{\prime}\left(  0\right)   =1
\end{cases}
\]
and where $d\theta^{2}$ denotes the usual metric in $\mathbb{S}^{n-1}$. We
also assume that the norm of the Killing field does not depend on $\theta$, so that
\[
|Y|^{2}=e^{-2\psi(r)},
\]
In this case, the ambient metric $\g_{M}$ of $M=P\times_{e^{-\psi}}\R $ is
written in terms of cylindrical coordinates $(s,r,\theta)$ as
\[
\g_{M}=e^{-2\psi(r)}ds^{2}+dr^{2}+\xi^{2}(r)d\theta^{2}.
\]
and $M$ is a doubly-warped product with respect to warping functions of the
coordinate $r$. The smoothness of $\psi$ implies that  the pole $o$ is a critical point for $e^{-\psi}$,
namely 
\[
\frac{d e^{-\psi(r)}}{dr}\left(  0\right)= -e^{-\psi(r)}\frac{d \psi(r)}{dr}=0.
\]
A rotationally invariant Killing graph $\Sigma_{0}\subset M$ is defined by a function $u$ that depends only on the radial coordinate $r$. In this case (\ref{capillary-3}) becomes
\begin{equation}
\label{H-ode}
\bigg(\frac{u'(r)}{W}\bigg)' + \frac{u'}{W} \big(\Delta_P r - \langle \nabla^P \psi, \nabla^P r\rangle\big)= nH,
\end{equation}
where
\[
W = \sqrt{e^{2\psi(r)}+u'^2(r)}
\]
and $'$ denotes derivatives with respect to $r$. Note that the weighted Laplacian of  $r$ is given by
\[
\Delta_P r- \langle \nabla^P \psi, \nabla^P r\rangle = -\psi'(r) + (n-1) \frac{\xi'(r)}{\xi(r)} = \frac{|Y|'(r)}{|Y|(r)} + (n-1) \frac{\xi'(r)}{\xi(r)},
\]
which is, up to a factor $1/n$, the mean curvature $H_{{\rm cyl}}(r)$ of the cylinder over the geodesic sphere of radius $r$ centered at $o$ and ruled by the flow lines of $Y$ over that sphere. We also have
\begin{equation}
\label{H-Hpsi}
nH_{\psi} = nH -\frac{u'(r)}{W}\langle\nabla^P \psi, \nabla^P r\rangle = nH -\psi'(r)\frac{u'(r)}{W} = {\rm div}^P_{2\psi}\bigg(\frac{u'(r)}{W}\nabla^P r\bigg)\cdot
\end{equation}
It follows from (\ref{H-ode}) that both $H$ and $H_{\psi}$ depend only on $r$. Integrating both extremes in (\ref{H-Hpsi}) against the weighted measure ${\rm d}P_\psi = e^{-\psi} {\rm d}P$ one obtains in this particular setting a first-order equation involving $u(r)$, namely
\begin{equation}
\label{H-rot-flux}
\frac{u'(r)}{W} e^{-2\psi(r)}\xi^{n-1}(r) = \int_0^r nH_\psi e^{-2\psi(\tau)} \xi^{n-1}(\tau)\,{\rm d}\tau.
\end{equation}
Denoting the right hand side in (\ref{H-rot-flux}) by $I(r)$ and solving it for $u'(r)$ yields
\begin{equation}
\label{u-prima}
u'^2= e^{2\psi} \frac{I^2 }{e^{-4\psi}\xi^{2(n-1)}-I^2}
\end{equation}
We assume that $u'(r)\le 0$ and denote
\begin{equation}
V_\psi(r) = \frac{1}{|\mathbb{S}^{n-1}|}{\rm vol}_\psi (B_r(o)) =\int_0^r e^{-2\psi(\tau)}\xi^{n-1}(\tau)  \, {\rm d}\tau 
\end{equation}
 and
\[
A_\psi(r) = \frac{1}{|\mathbb{S}^{n-1}|}{\rm vol}_\psi (\partial B_r(o)) = e^{-2\psi(r)}\xi^{n-1}(r).
\]

Fixed this geometric setting, we  prove the existence of compact rotationally symmetric Killing graphs with constant weighted mean curvature.
\begin{theorem}
\label{H-rot-existence} Suppose that the ratio $\frac{A_\psi(r)}{V_\psi(r)}$ is non-increasing for $r\in (0,R)$. Let $H_0$ be a constant with
\begin{equation}
\label{isop-ineq}
n|H_0| =  \frac{A_\psi(r_0)}{V_\psi(r_0)} 
\end{equation}
for some $r_0\in (0,R)$.
Then there exists a compact rotationally symmetric Killing graph $\Sigma_0\subset P\times_{e^{-\psi}}\mathbb{R}$ of a radial function $u(r)$, $r\in [0,r_0]$, given by
\begin{equation}
\label{u-solution}
u(r) = \int^r_{r_0} e^{\psi(\tau)} \frac{I(\tau)}{\sqrt{e^{-4\psi(\tau)}\xi^{2(n-1)}(\tau)-I^2(\tau)}} \, {\rm d}\tau.
\end{equation}
with constant weighted mean curvature $H_\psi=H_0$ and boundary $\partial B_{r_0}(o)\subset P$. The weighted height function in this graph is bounded as follows
\begin{equation}
\label{u-estimate}
e^{-\psi(r)} u(r) \le 
e^{\sup_{B_r(o)} \psi-\inf_{B_r(o)}\psi}\int_{0}^{r_{0}}\frac{-nH_0}{\sqrt{\frac{A^2_\psi(\tau)}{V^2_\psi(\tau)}-n^2 H^2_0}}d\tau.
\end{equation}
\end{theorem}

\begin{proof} Since $A_\psi(r)/V_\psi(r)$ is non-increasing it follows from (\ref{isop-ineq}) that 
\[
e^{-4\psi(r)}\xi^{2(n-1)}(r)-I^2(r) \ge A_\psi^2(r) - n^2 H^2_\psi V_\psi^2(r)\ge 0
\]
for $r\in (0,r_0]$ what guarantees that the expression
\begin{equation}
\label{u-rot}
u(r) = \int_{r_0}^{r} e^{\psi(\tau)} \frac{I(\tau)}{\sqrt{e^{-4\psi(\tau)}\xi^{2(n-1)}(\tau)-I^2(\tau)}} \, {\rm d}\tau
\end{equation}
is well-defined for $r\in [0,r_0]$ with $u'(r)\le 0$. For further reference, we remark an application of L'H\^opital's rule shows that 
\begin{equation}
\label{h-cyl}
\lim_{r\to 0} \frac{A_\psi(r)}{V_\psi(r)}  =\lim_{r\to 0} \bigg(2\frac{|Y|'(r)}{|Y|(r)} + (n-1) \frac{\xi'(r)}{\xi(r)}\bigg)\cdot
\end{equation}
In order to get a more detailed analysis at $r=r_0$ we consider a parametrization of $\Sigma_0$ in terms of cylindrical coordinates as
\[
(t,\theta) \mapsto (r(t), s(t), \theta),
\]
where $t$ is the arc-lenght parameter defined by
\[
\dot r^2 (t)+ e^{-2\psi(r(t))} \dot s^2(t) =1
\]
and $\cdot$ denotes derivatives with respect to $s$. Since 
\[
u'(r(t)) = \frac{\dot s(t)}{\dot r(t)}
\]
we impose $\dot s\ge 0$ and $\dot r\le 0$.
Hence $W=-e^{\psi}/\dot r$
whenever $\dot r(s)\neq 0$ and  (\ref{H-rot-flux}) is written as
\begin{equation}
\label{flux-param}
\dot s  e^{-3\psi(r)}\xi^{n-1}(r) = -\int_0^r nH_\psi e^{-2\psi(\tau)} \xi^{n-1}(\tau)\,{\rm d}\tau
\end{equation}
Since $\xi(0)=0$, $\xi'(0)=1$ and  $\psi'(0)=0$, applying L'H\^{o}pital's rule as above shows that
\[
\lim_{r\to 0}\frac{\int_0^r nH_\psi e^{-2\psi(\tau)} \xi^{n-1}(\tau)\,{\rm d}\tau}{e^{-3\psi(r)}\xi^{n-1}(r)} = \lim_{r\to 0}\frac{nH_\psi e^{\psi(r)}}{-3\psi'(r) + (n-1)\frac{\xi'(r)}{\xi(r)}} =0
\]
and we conclude that $e^{-\psi(r)}\dot s\to 0$ as $r\to 0$. Therefore $\dot s \to 0$ as $r\to 0$ what implies that $\Sigma_0$ is smooth at its intersection with the vertical axis of revolution. 

Now, we have from (\ref{flux-param}) and (\ref{isop-ineq}) that $\dot r =0$ at $r=r_0$ since $e^{-\psi(r_0)}\dot s(r_0) =1$. Finally, let $\phi$ be the angle between a meridian $\theta={\rm cte.}$ in $\Sigma_0$ and  and radial vector field $\partial_r$. We have
\[
\frac{u'(r)}{W} = -\frac{\dot s}{e^{\psi(r)}} = -\sin\phi.
\]
Hence
\begin{eqnarray*}
\bigg(\frac{u'(r)}{W} \bigg)' = -\cos\phi \,\dot \phi (t)\frac{1}{\dot r(t)} = -\dot \phi(t).
\end{eqnarray*}
On the other hand
\begin{eqnarray*}
\bigg(\frac{u'(r)}{W} \bigg)'  &= & nH - \frac{u'(r)}{W} \bigg(\frac{|Y|'(r)}{|Y|(r)}+(n-1)\frac{\xi'(r)}{\xi(r)}\bigg)\\
& = & nH_\psi + \sin \phi  \bigg(2\frac{|Y|'(r)}{|Y|(r)}+(n-1)\frac{\xi'(r)}{\xi(r)}\bigg)
\end{eqnarray*}
In sum, $\Sigma_0$ is parameterized by the solution of the first order system
\begin{equation}
\begin{cases}
\dot r & = \cos \phi \\
e^{-\psi (r)}\dot s & =\sin \phi\\
\dot \phi & =  -nH_0 - \sin \phi  \Big(2\frac{|Y|'(r)}{|Y|(r)}+(n-1)\frac{\xi'(r)}{\xi(r)}\Big),
\end{cases}
\end{equation}
with initial conditions $r(0) = r_0, s(0)=0, \phi(0)= \frac{\pi}{2}$. The height estimate follows directly from (\ref{u-solution}). This finishes the proof. \end{proof}



It is worth to point out that, in the classical situation of the Euclidean space where $e^{-\psi}=|Y|\equiv1$ and $\xi\left(  r\right)  =r$, the Killing graph defined by
$u(r)$ reduces to the standard sphere and (\ref{u-estimate})
gives rise to the expected sharp bound 
\[
u(r)  \leq \frac{1}{|H|}.
\]
Actually, a similar conclusion can be achieved if we choose $\psi\left(
r\right)  $ in such a way that 
\[
e^{-2\psi\left(  r\right)}  \xi^{n-1}\left(  r\right)  =r^{n-1}.
\]
Note that this choice is possible and compatible with the request
$d\psi/dr\left(  0\right)  =0$ because $\xi\left(  r\right)  $ is odd at
the origin.  This choice corresponds to the case when
\[
\frac{A_\psi(r)}{V_\psi(r)} = \frac{n}{r}
\]
as in the Euclidean space.
We have thus obtained the following height estimate

\begin{corollary}
\label{th_symm-example}Let $P=[0,R)\times_{\xi}\mathbb{S}^{n-1}$ be an
$n$-dimensional model manifold with warping function $\xi$ and let
$\psi:\left[0,R\right)  \rightarrow\R _{>0}$ be the smooth, even
function defined by 
\[
\psi(r)  =c\cdot
\begin{cases}
 \frac{(n-1)}{2}\log\frac{\xi (r)}{r} & r\neq0\\
1 & r=0,
\end{cases}
\]
where $c>0$ is a given constant. Fix $0<r_{0}<R$, let $H_{0}=-1/r_{0}$ and
define 
\[
u(r) =\int^{r_0}_{r}\frac{-H_{0}\tau}{\sqrt{1-H_0^2
\tau^{2}}}e^{\psi( \tau)} d\tau.
\]
Then, in the ambient manifold $M=P\times_{e^{-\psi}}\R $, the Killing graph of 
$u$ over $\Omega=B_{r_{0}}^{P}(
o) \subset P$ has  constant weighted mean curvature $H_\psi=H_0$ with respect to the upward pointing Gauss map.
Moreover, 
\[
0\leq e^{-\psi(r)}u(r) \leq\frac
{\max_{\lbrack0,r_{0}]}e^{-\psi}}{\min_{[0,r_{0}]}e^{-\psi}}\cdot
\frac{1}{|H_0|}=e^{\sup_\Omega \psi-\inf_\Omega\psi}\frac{1}{|H_0|}.
\]
\end{corollary}

\vspace{3mm}

The counterpart of Theorem \ref{H-rot-existence} and Corollary \ref{th_symm-example} in the case of constant mean curvature  can be obtained along the same lines by integrating both sides in (\ref{capillary-3}) instead of (\ref{H-Hpsi}). Denoting
\[
V(r) = \frac{1}{|\mathbb{S}^{n-1}|}{\rm vol} (B_r(o)) =\int_0^r e^{-\psi(\tau)}\xi^{n-1}(\tau)  \, {\rm d}\tau
\]
and
\[
A(r) = \frac{1}{|\mathbb{S}^{n-1}|}{\rm vol} (\partial B_r(o)) = e^{-\psi(r)}\xi^{n-1}(r)
\]
one obtains
\begin{theorem}
\label{Hw-rot-existence} Suppose that the ratio $\frac{A(r)}{V(r)}$ is non-increasing for $r\in (0,R)$. Let $H_0$ be a non-positive constant with
\begin{equation}
\label{isop-ineq}
n|H_{0}| =  \frac{A(r_0)}{V(r_0)} 
\end{equation}
for some $r_0\in (0,R)$.
Then there exists a compact rotationally symmetric Killing graph $\Sigma_0\subset P\times_{e^{-\psi}} \mathbb{R}$ of a radial function $u(r)$, $r\in [0,r_0]$, given by
\begin{equation}
\label{u-solution}
u(r) = \int^r_{r_0}  \frac{nH_{0}V(\tau)}{\sqrt{A^2(\tau)-n^2 H_{0}^2 V^2(\tau)}} \, {\rm d}\tau.
\end{equation}
with constant mean curvature $H=H_{0}$ and boundary $\partial B_{r_0}(o)\subset P$. The height function in this graph is bounded as follows
\begin{equation}
\label{u-estimate}
 e^{-\psi(r)}u(r) \le 
e^{\sup_{B_r(o)} \psi-\inf_{B_r(o)}\psi}
\int_{0}^{r_{0}}\frac{-nH_{0}}{\sqrt{\frac{A^2(\tau)}{V^2(\tau)}-n^2 H_{0}^2}}d\tau.
\end{equation}
\end{theorem}

\vspace{3mm}

The analog of Corollary \ref{th_symm-example} in the case of constant mean curvature is 

\begin{corollary}
\label{th_symm-example-w}Let $P=[0,R)\times_{\xi}\mathbb{S}^{n-1}$ be an
$n$-dimensional model manifold with warping function $\xi$ and let
$\psi:\left[0,R\right)  \rightarrow\R _{>0}$ be the smooth, even
function defined by 
\[
\psi(r)  =c\cdot
\begin{cases}
 (n-1)\log\frac{\xi (r)}{r} & r\neq0\\
1 & r=0,
\end{cases}
\]
where $c>0$ is a given constant. Fix $0<r_{0}<R$, let $H_{0}=-1/r_{0}$ and
define 
\[
u(r) =\int^{r}_{r_0}\frac{H_{0}\tau}{\sqrt{1-H_{0}^2
\tau^{2}}}d\tau.
\]
Then, in the ambient manifold $M=P\times_{e^{-\psi}}\R $, the Killing graph of 
$u$ over $\Omega=B_{r_{0}}^{P}(
o) \subset P$ has  constant mean curvature $H_{0}$ with respect to the upward pointing Gauss map.
Moreover, 
\[
0\leq 
e^{-\psi(r)}
u(r) \leq
\frac{\max_{\lbrack0,r_{0}]}e^{-\psi}}{\min_{[0,r_{0}]}e^{-\psi}}\cdot
\frac{1}{|H_{0}|}
=e^{\sup_\Omega \psi-\inf_\Omega\psi}\frac{1}{|H_0|}.
\]
\end{corollary}

\begin{remark}
\rm{Up to the factor $2$ in the exponential, this is  the estimate obtained in Theorem \ref{th_fheightestimate}
above.  We suspect that the high rotational symmetry considered in the example prevents to achieve the maximum height predicted by the theorem. On the other hand we conjecture that the rotationally symmetric graphs can be used as barriers to obtain sharp estimates in the case of general warped spaces $P\times_{e^{-\psi}}\mathbb{R}$  in the case when the radial sectional curvatures of $P$ are bounded from above by some radial function.}
\end{remark}

\end{document}